\documentclass[12pt]{article}

\usepackage[a4paper, margin=2cm]{geometry}

\usepackage{amssymb}
\usepackage{amsmath,amsthm}
\usepackage{epstopdf}
\usepackage{tikz}
\usepackage{tkz-graph}
\usepackage{mathtools}
\usepackage{csquotes}
\usepackage[bb=boondox]{mathalfa} 

\usepackage{tikz,subcaption,amsmath}

\usetikzlibrary{arrows.meta,positioning}

\newtheorem{theorem}{Theorem}[section]
\newtheorem{corollary}[theorem]{Corollary}
\newtheorem{lemma}[theorem]{Lemma}
\newtheorem{proposition}[theorem]{Proposition}

\newtheorem{Conjecture}[theorem]{Conjecture}

\newtheorem{example}{Example}

\def\qed{\vbox{\hrule
  \hbox{\vrule\hbox to 5pt{\vbox to 8pt{\vfil}\hfil}\vrule}\hrule}}

\def\endproof{\unskip \nobreak \hskip0pt plus 1fill \qquad \qed \par \vspace{0.15cm}}
\newcommand{\reals}{\mbox{${\rm I}\!{\rm R}$}}

\newcommand{\diag}{\mbox{\rm diag}}

\newcommand{\conv}{\mbox{\rm conv$\,$}}

\newcommand{\mb}{\mathbb}
\newcommand{\mc}{\mathcal}

\begin{document}

\title{Doubly Stochastic Matrices and Modified Laplacian Matrices  of Graphs}

\author{
 Enide Andrade\footnote{Center for Research and Development in Mathematics and Applications,
 Department of Mathematics,  University of Aveiro, Portugal. {\tt enide@ua.pt}.}\\
 Geir Dahl\footnote{Department of Mathematics,   University of Oslo, Norway. {\tt geird@math.uio.no.}Corresponding author.} 
 }

\maketitle


\date{}


\begin{abstract}
We consider modified Laplacian matrices of graphs, obtained by adding the identity matrix to the Laplacian matrix $L_G$ of a graph $G$. This results  in a positive definite matrix $\tilde{L}_G$. The inverse of $\tilde{L}_G$ is a doubly stochastic matrix. The goal of this paper is to investigate this inverse matrix and how it depends on properties  of the underlying graph $G$. In particular, we introduce a general monotonicity property for the entries of the inverse, and derive a sharper version for the case of path graphs. Finally, we show that, in the case of a path graph, the entries of the inverse can be expressed in terms of Fibonacci numbers via an $LU$ factorization.
We also establish a lower bound for the diagonal entries of this inverse for a tree as a function of the distances between vertices. Furthermore, we present a simple and efficient  algorithm for computing the inverse when the graph is a tree. Moreover, for a general graph, we show that the diagonal entries of this inverse is strictly largest in each row and column. Finally, we discuss a connection to partial differential equations, such as the heat equation.

 
\end{abstract}



\bigskip\bigskip\bigskip

\section{Introduction} 
\label{sec: introduction}

This paper investigates  doubly stochastic matrices that arise as the inverses of modified Laplacian matrices associated with graphs. Thus, associated to a graph $G$ there is a doubly stochastic matrix $B_G$, and our goal is to study this map. 

The Laplacian matrix of a simple (undirected) graph, denoted $L_G$, is a fundamental concept in graph theory and linear algebra that encodes the structure of a graph in matrix form. It plays a crucial role in analyzing various properties of graphs, such as connectivity (for example, detecting connected components and identifying bottlenecks), flow dynamics (modeling phenomena like diffusion, heat transfer, or information spread through random walks), and spectral characteristics (where eigenvalues and eigenvectors yield insights into the the graph structure and form a basis for graph partition methods); see, for instance, \cite{Fiedler_alg_conn, Fiedler2, Goerttler,Molitierno11}. The Laplacian matrix $L_G$ is defined as the difference between the diagonal matrix of vertex degrees and the adjacency matrix of $G$. Comprehensive surveys of the properties and applications of Laplacian matrices can be found in \cite{CHUNG, Merris94, Molitierno11}.

Doubly stochastic matrices are (componentwise) nonnegative matrices where each row and column sum is $1.$
These matrices are of interest in several areas of mathematics, such as in combinatorics, optimization (the assignment problem) and matrix theory (for instance, majorization theory and matrix scaling), see \cite{RAB06,MaOlAr11}. Moreover, the set $\Omega_n$ of $n \times n$ doubly stochastic matrices is a polytope, often referred to as the {\em Birkhoff polytope}, and it is a well-studied object with interesting properties (see \cite{RAB06} for an in-depth study and many references). 

The motivation for this paper came from the work by Merris \cite{Merris}, where a connection between graph Laplacians and doubly stochastic matrices  was established and studied; the author introduced  the mentioned modified Laplacian and established some properties of the inverse matrix which is  doubly stochastic. Another motivation is the role played by graph Laplacians in connection with certain partial differential equations, such as the heat equation. In such settings the modified Laplacian may be be of interest, and we discuss this briefly at the end of the paper.
\medskip

 Let $G$ be a  simple (undirected) graph with  $n$ vertices $v_1, v_2, \ldots, v_n$. Sometimes we identify each vertex $v_i$ by its label $i$ ($i \le n$). Let  $L_G$ be the Laplacian matrix of $G$, so $L_G=D-A$ where $A$ is the adjacency matrix of $G$ and $D$ is the diagonal matrix with diagonal given by the degrees of the vertices. Define the modified graph Laplacian 
 \[
    \tilde{L}_G=   L_G+I_n.
 \]
 This matrix is symmetric, positive definite and therefore invertible. 
 We let the $i$'th row and $i$'th column of $\tilde{L}_G$ be associated with vertex $v_i$ ($i \le n$). 
 
 The matrix $\tilde{L}_G$ and its inverse were introduced and studied in \cite{Golender,Merris}, see  also \cite{Chebotarev}.  Let $G'$ be the graph obtained from $G$ by adding a vertex $w$ and an edge between $w$ and each vertex $v_i$. As noted in \cite{Merris} one may view $\tilde{L}_G$ as the principal submatrix obtained from the Laplacian matrix $L_{G'}$ of $G'$ by deleting the row and column associated with $w$. 
 
 The inverse matrix of $\tilde{L}_G$, denoted by $B_G$, will be the focus of this paper, so  
 \begin{equation}
 \label{eq:define_B}
   B_G=\tilde{L}_G^{-1}=\left(L_G+I_n\right)^{-1}.
 \end{equation}  
 As $\tilde{L}_G$ is an M-matrix, its inverse $B_G$ is entrywise nonnegative, in fact, positive (see below). Moreover, as the all ones vector $e$ is an eigenvector associated to the eigenvalue $0$ of $L_G$, $e$ is an eigenvector for $B_G$ associated to the eigenvalue $1$, that is, $B_G$ is doubly stochastic. 
 
 This inverse matrix $B_G=[b_{ij}]$ is therefore associated with the graph $G$, and we will investigate how the elements in $B_G$ reflect and depend on properties in $G$. $B_G$ was  called the {\em doubly stochastic matrix of $G$} in \cite{Merris,Merris_II,Zhang05}. $B_G$ is symmetric. 
 We associate rows and columns of $B_G$ to vertices as we did for $\tilde{L}_G$, so the $i$'th row and $i$'th column of $B_G$ are associated with vertex $v_i$ ($i \le n$). Each column $x$ of the inverse matrix $B_G$ is the solution of a system of linear equations
 \begin{equation}
     \label{eq:Laplace_eq}   \tilde{L}_G\, x=e_k,
 \end{equation}
 where $e_k$ is the $k$'th unit vector, so $x$ is column $k$ in $B_G$. These  equations will be called the {\em Laplacian equations}, and they will be used to derive properties of $B_G$.

It was shown in \cite{Merris} (see also \cite{Chebotarev}) that if $G$ is connected, and the Laplacian eigenvalues of $G$ are $\lambda_1 \geq \lambda_2 \geq \cdots \ge  \lambda_{n}= 0$, then $B_G$ is a positive definite doubly stochastic matrix that is entrywise positive,
 with eigenvalues 
\[
1\ge\frac{1}{1+\lambda_{n-1}}\ge \cdots\ge \frac{1}{1+\lambda_{1}}>0.
\]

This matrix was studied in \cite{Merris_II} and some of its properties  were established, expressing some relations of its entries and the structure of the graph. In fact $\omega(G)>0$ if and only if $G$ is connected, and 
$\omega(G) \leq 1/(n + 1)$ with equality if and only if $G$ is the complete graph.  If $i \neq j$, 
then 
\[
\omega(G) \geq  2b_{ij},
\]
with equality if and only if the $i$'th vertex has degree $n - 1.$ 
Additionally it was suggested that $\mbox{\rm{max}}_{i,j}\,  b_{ij}$ identifies the most remote vertices of $G,$ in the sense that they are ``less connected". This supported an assertion in \cite{Chebotarev}. In \cite{Chebotarev}, the authors deal with multigraphs (and multidigraphs) and propose a family of graph structural indices related to the matrix-forest theorem. They define the product of the weights of all edges that belong to a subgraph $H$ of a multigraph $G$ as the weight of conductance of $H$ and denoted by $ \xi (H)$. Additionally, for every nonempty set of subgraphs $G$, its weight is defined as $\xi(G) = \sum_{H \in G} \xi (H).$

Then, the entries $b_{ij}$ of $B_G$  are identified in \cite[Theorem 3]{Chebotarev} (called matrix-forest theorem) for weighted multigraph $G$ as 
\[
b_{ij} =\frac{\xi (\mathcal{F})^{ji}}{\xi(\mathcal{F})},
\]
where $\mathcal{F}$ is the set of  set of all spanning rooted forests and $ \mathcal{F}^{ij}$ denotes the set of those spanning rooted forests of $G$ such that $v_i$ and $v_j$ belong to the same tree rooted at $v_i$. For weighted multidigraphs, the entries $b_{ij}$ were also expressed in a similar way.
Additionally, it was stated in \cite{Chebotarev} that, the matrix-forest theorem allows to consider the matrix $B_G$ as the matrix of ``relative forest
accessibilities" of the vertices of a multigraph $G$ (or a multidigraph). These values can be used to measure the proximity between vertices (the ``farther" $v_i$ from $v_j$, the smaller is $b_{ij}$).

Recently, in \cite{Zhang} it was proven that, for a tree $T$ with $n$ vertices, one has
$$\frac{\sqrt{5}}{(\frac{3+\sqrt{5}}{2})^{n} -(\frac{3-\sqrt{5}}{2})^{n}}\leq \omega(T) \leq \frac{1}{2(n+1)},$$
with right equality if and only if $T$ is a star and left equality if and only if $T$ is a path. 
We will see later that if $T$ is a path of order $n$, denoted by
$P_n$, then the expression on the left-hand side is the smallest entry of 
$B_{P_{n}}$ and, it can be expressed as the ratio of the first Fibonacci number to the 
$2n$-th Fibonacci number in the Fibonacci sequence.

\noindent Let $\lambda_{n-1} =a(G)$; the algebraic connectivity of $G$. The second largest eigenvalue of $B_G$ is $1/(1 + a(G)).$ In \cite{Merris_II} some relations between these new graph invariants and the algebraic connectivity were established. In particular, the following two conjectures were presented:

\begin{Conjecture}\cite{Merris_II}\label{1.2}
Let $G$ be a graph on $n$ vertices. Then
$$a(G) \geq  2(n + 1)\omega(G).$$
\end{Conjecture}
\begin{Conjecture} \cite{Merris_II}\label{1.1}
 Let $E_n$ be the degree anti-regular graph, that is, the unique connected graph whose
vertex degrees attain all values between $1$ and $n-1.$ Then
$\omega(E_n) = 1 /2(n + 1).$
\end{Conjecture}

\noindent Berman and Zhang in \cite{Berman} confirmed Conjecture \ref{1.2}. In \cite{Zhang} a counterexample to Conjecture \ref{1.1} was given. 

The Laplacian matrix $L_G$ is singular, so it has no inverse. Still, it has a group inverse (generalized inverse) and this matrix has been investigated, see Chapter 8 in \cite{Molitierno11} and the references given there. In particular, one studies this for the weighted Laplacian matrix and with combinatorial results for trees. 

The remaining paper is organized as follows. Section 2 presents monotonicity properties of the entries of $B_{P_n}$ for a path $P_n$ with $n$ vertices. It is further shown that these entries can be expressed in terms of quotients of Fibonacci numbers via an $LU$ decomposition. Section 3 establishes a monotonicity property for a general graph $G$ showing that certain results extend as strong inequalities. For a tree $T$ a lower bound on the diagonal entries of $B_T$ is presented in terms of vertex distances, and this bound is explicitly determined for a specific family of trees.
Section 4 gives an algorithm for computing the entries of $B_T$ when $T$ is a tree, and  some applications are presented. Finally, Section 5 explores a connection to partial differential equations, focusing in particular on the heat equation.

\smallskip
{\bf Notation:} In a graph $G$ we write  $u \sim v$ to indicate that two vertices $u$ and $v$ are adjacent, i.e., $uv$ is an edge in $G$. We consider vectors in $\mb{R}^n$ as column vectors and identify these with  real $n$-tuples. The $i$'th component of a vector $x \in \mb{R}^n$ is usually denoted by $x_i$ ($ i \le n$). A vector $d=(d_1, d_2, \ldots, d_n) \in \mb{R}^n$ is called {\em monotone} if it is nonincreasing, i.e., $d_1 \ge  d_2 \ge  \cdots \ge  d_n$. 
A zero matrix, or vector, is denoted by $O$, and an all ones vector is denoted by $e$. Here the dimension is usually suppressed. The diagonal of an $n \times n$  matrix $A=[a_{ij}]$ is denoted by $\diag(A)$, so $\diag(A)=(a_{11}, a_{22}, \ldots, a_{nn})$. The transpose of a matrix $A$ is denoted by $A^t$. 
 Moreover, $P_n$ (resp. $S_n$) is the path (resp. the star) with $n$ vertices, and $K_n$ represents the complete graph with $n$ vertices. We use $\diamond$ to indicate the end of an example.


\section{Monotonicity for paths}
 \label{se:mono}

In this section we prove some interesting monotonicity properties of the inverse matrix $B_G=\tilde{L}_{G}^{-1}$ for paths ($G$ is a path). In the next section we consider more general graphs and show related properties.

Consider the path $P_n$ with $n$ vertices. Let $L_{P_n}$ be its Laplacian matrix and define $\tilde{L}_{P_n}=L_{P_n}+I_n$ where $I_n$ is the identity matrix of order $n$. 
For instance, if $n=4$, the matrix $\tilde{L}_{P_4}$ is
\[
\tilde{L}_{P_4}=
\left[
\begin{array}{rrrr}
   2 & -1 & 0 & 0 \\
    -1&  3 & -1 & 0  \\
    0 & -1&  3 & -1   \\
   0 & 0 & -1&  2 
\end{array}
\right].
\]

So $\tilde{L}_{P_n}$ is symmetric and positive definite, and these properties also hold for the inverse $\tilde{L}_{P_n}^{-1}$. We shall establish a certain monotonicity property of the inverse matrix 
$B_{P_n}=\tilde{L}_{P_n}^{-1}$.

Let $\tilde{L}_{P_n}=[a_{ij}]$, so  $a_{i,i+1}=a_{i+1,i}=-1$ ($1 \le i \le n-1$), $a_{ii}=3$ ($2\le i \le n-1$), and $a_{11}=a_{nn}=2$. Let $2\le k \le n-1$ and define the $k \times k$ matrix $M_k$ based on the matrix $\tilde{L}_{P_k}$ by changing the entry in position $(1,1)$ to 3 (this entry is 2 in $\tilde{L}_{P_k}$). We also define $M_1=[2]$. Similarly, define $\tilde{M}_k$ based on the matrix $\tilde{L}_{P_k}$ by changing the entry in position $(k,k)$ to 3 (this entry is 2 in $\tilde{L}_{P_k}$). Moreover $\tilde{M}_1=M_1$. The matrix $\tilde{M}_k$ is obtained from $M_k$ by 
reversing the order of both rows and columns, and therefore $\det \tilde{M}_k=\det M_k$.

\begin{example}
{\rm Considering $\tilde{L}_{P_4}$, we have:
\[
M_4=
\left[
\begin{array}{rrrr}
   3 & -1 & 0 & 0 \\
    -1&  3 & -1 & 0  \\
    0 & -1&  3 & -1   \\
   0 & 0 & -1&  2 
\end{array}
\right] \;\;\mbox{and}\; \;
\tilde{M}_4=\left[
\begin{array}{rrrr}
   2 & -1 & 0 & 0 \\
    -1&  3 & -1 & 0  \\
    0 & -1&  3 & -1   \\
   0 & 0 & -1&  3 
\end{array}
\right].
\]
Note that $\det M_2=5$ and $\det M_1=2$.} \hfill{$\diamond$}
\end{example} 

\noindent For an $n \times n$  matrix $A$, let $A_{[ij]}$ denote the $(n-1) \times (n-1)$  submatrix obtained from $A$ by deleting row $i$ and column $j$. 

\begin{lemma}
 \label{lem:det-recursive}
 \begin{equation}
  \label{eq:double-property}
     \det M_k > 2  \det M_{k-1} \;\;\;(2\le k \le n-1).
 \end{equation}
\end{lemma}

\begin{proof}
  We prove the result by induction on $k$, so assume (\ref{eq:double-property}) holds for integers smaller than $k$. The inequality holds for $k=2$. We compute $\det M_k$ by Laplace expansion  in the first row. Observe that for $k \ge 3$ $(M_k)_{[12]}$ is a $2 \times 2$ block upper triangular matrix with determinant 
  $(-1)\cdot \det M_{k-2}$. This gives
 \[
  \begin{array}{ll} \vspace{0.1cm}
    \det M_k &=3\det M_{k-1} - (-1) \cdot (-1) \cdot \det M_{k-2} \\ \vspace{0.1cm}
              &=3\det M_{k-1} -  \det M_{k-2}  \\ \vspace{0.1cm}
              &=2\det M_{k-1}+(\det M_{k-1}-  \det M_{k-2}) \\ \vspace{0.1cm}
              &>2\det M_{k-1},
  \end{array}  
 \]
 where the inequality holds due to  the induction hypothesis. This proves (\ref{eq:double-property}) by induction.
\end{proof}

\medskip
Let $A=[a_{ij}]$ be an $n \times n$ matrix. We say that $A$ is {\em $d$-monotone} ($d$ indicates diagonal) if 
\[
     a_{i1}< a_{i2}<\cdots <a_{i,i-1} < a_{ii} > a_{i,i+1} > \cdots > a_{in} \;\;\;(i \le n).
\]
Thus, in each row the largest entry is on the (main) diagonal, and from the diagonal the entries strictly decrease in each direction. If $A$ in addition is symmetric, the same property holds in every column. 

\begin{theorem}
 \label{thm:B_monotone}
  Consider the path $P_n$ and the matrix 
  $B_{P_n}=\tilde{L}_{P_n}^{-1}=[b_{ij}]$. 
 Then $B_{P_n}$ is $d$-monotone. In fact,  the following stronger property holds 
 \begin{equation}
 \label{eq:d-mon-extra}
 \begin{array}{cl} \vspace{0.1cm}
     b_{ij} \ge 2 \,b_{i,j+1}   &(1\le i\le j \le n-1), \\ \vspace{0.1cm}
     b_{ij} \ge 2 \,b_{i,j-1}    &(2\le j\le i \le n).
 \end{array}    
 \end{equation}
\end{theorem}

\begin{proof}
 Consider row $i$ of $B_{P_n}$. For simplicity let us just write $\tilde{L}_n=\tilde{L}_{P_n}$. From the matrix equation $\tilde{L}_{n} B_{P_n}=I_n$, the $i$'th column of $B_{P_n}$, call it $x$, satisfies $\tilde{L}_n x=e_i$ where $e_i$ is the $i$'th unit vector. As $B_{P_n}$ is symmetric, row $i$ of $B_{P_n}$ equals $x^t$ (the transpose of $x$). So, we want to solve  $\tilde{L}_n x=e_i$. Let the $j$'th component of $x$ be denoted by $x_j$ and let $\tilde{L}_n[j;e_i]$ be the matrix obtained from $\tilde{L}_n$ by replacing the $j$'th column by $e_i$.  By Cramer's rule
 \[
      x_j=\frac{\det \tilde{L}_n[j;e_i]}{\det \tilde{L}_n}.
 \]
Consider $j \geq i$. Then 
\[
(\tilde{L}_n[j;e_i])_{[ij]}=
\left[
\begin{array}{ccc}
\tilde{M}_{i-1} & * & * \\ \vspace{0.2cm}
 O    & A_2 & * \\ \vspace{0.2cm}
 O    & O & M_{n-j}
\end{array}
\right],
\]
where $A_2$ is an upper triangular $(j-i) \times (j-i)$  matrix with each diagonal element equal to $-1$ (this matrix is void is $j=i$), and $*$ indicates some submatrix whose entries play no role here. \\

%

We obtain
\[
 \begin{array}{ll} \vspace{0.2cm}
   \det \tilde{L}_n[j;e_i] &= 
   (-1)^{i+j} \cdot \det (\tilde{L}_n[j;e_i])_{[ij]} \\ \vspace{0.2cm}
   &=(-1)^{i+j} (-1)^{j-i} \det M_{i-1}   \det M_{n-j} \\ \vspace{0.2cm}
   &= \det M_{i-1}   \det M_{n-j} \\ \vspace{0.2cm}
   &\ge \det M_{i-1}  \cdot 2 \det M_{n-(j+1)} \\ \vspace{0.2cm}
   &=2\det \tilde{L}_n[j+1;e_i] 
 \end{array}, 
\]
due to Lemma \ref{lem:det-recursive}. This implies that $x_j\ge 2x_{j+1}$. (Note that $\det \tilde{L}_n>0$ as this matrix is positive definite.) Repeating this argument we get the first inequality in (\ref{eq:d-mon-extra}). 
  Finally,  the proof of the second inequality in (\ref{eq:d-mon-extra}), for the case  $j \le i$, is very similar, so we omit it. 
\end{proof}

The inequalities in (\ref{eq:d-mon-extra}) show that the entries in each row of $B_{P_n}$ decay exponentially fast when one moves away from the diagonal, with a decay factor of at least $2$.

Some inequalities in (\ref{eq:d-mon-extra}) hold with equality as we have 
\[
\begin{array}{rcl} \vspace{0.2cm}
   b_{i,n-1}&=2\,b_{in}   &(1\le i \le n-1), \\  \vspace{0.2cm}
   b_{i2}&=2\,b_{i1}   &(2\le i \le n). 
\end{array}
\]   
In fact, these equations are a consequence of a result below (Lemma \ref{lem:pendant}).

We shall determine an $LU$ factorization of the matrix $\tilde{L}_{P_n}$ for the path $P_n$. Next we use this factorization to compute the matrix $B_{P_n}=\tilde{L}_{P_n}^{-1}$.

Note that $\tilde{L}_{P_n}$ is a tridiagonal matrix. Explicit formulas for the elements of the inverse of a general tridiagonal matrix can be seen for instance in \cite{Fonseca, Lewis, Mallik}. The approach in \cite{Mallik} is based on linear difference equations.  In \cite{Mikkawy}, the Doolitle $ LU$ factorization is studied and the author presented an algorithm to study the inverse. Using this approach, and considering $\tilde{L}_{P_n} = L U$, we present here explicitly the entries of the matrices $L$ and $U$ whose entries will be expressed as quotients of Fibonacci numbers. \\

We first find the $L_1U$ factorization of $\tilde{L}_{P_n}.$ Before that let us recall the Fibonacci sequence. 
Let $f_n$ denote the $n$'th Fibonacci number. So $f_1=f_2=1$ and 
\begin{equation} \label{Fibonacci}
f_n=f_{n-1}+f_{n-2},
\end{equation}
for each $n\ge 3$, and an explicit formula for $f_n$ is well known. 

%

\begin{proposition} Let $P_n$ be a path with $n$ vertices. Then, the $L_1U$ factorization of $\tilde{L}_{P_n}$ is, 
$\tilde{L}_{P_n} = L_1 U$, where 
\begin{eqnarray} \label{LandU}
L_1= \left[
  \begin{array}{ccccc}
  1    & 0      &            & \cdots   & 0 \\
  x_1  & 1      &            &          & 0 \\
  0    & x_2    & 1          &          & 0 \\
\vdots &        & \ddots     & \ddots   & \vdots \\
0      & 0      & 0          & x_{n-1} & 1
\end{array}
  \right], \;
U= \left[
  \begin{array}{ccccc}
  y_1    & -1      &            & \cdots     & 0 \\
  0      & y_2     & -1          &             & 0 \\
\vdots & \vdots & \ddots     &              & \vdots \\
0      & 0      &            &     y_{n-1}  &  -1\\
0      & 0      & 0          & 0            & y_n
\end{array}
  \right]
\end{eqnarray}
with 

\[
\begin{array}{ll}
 (i) & x_1  =  -1/2; \, y_1 =  2, \\
  (ii) &x_{i}=   - f_{2i-1}/f_{2i+1}, \mbox{ and} \;\; 
        y_i  =   f_{2i+1}/f_{2i-1} \;\mbox{ for $2 \leq i \leq n-1$,}\\
(iii) &y_n  =   f_{2n}/f_{2n-1}.
\end{array}
\]
\end{proposition}

\begin{proof}
Suppose that $\tilde{L}_{P_n}=L_{1}U$, with $L_{1}$ and $U$ as in (\ref{LandU}).  Analyzing the equality of the entries $\tilde{L}_{P_n}=L_{1}U$ we get from:
\[
\begin{array}{lll}
\mbox{\rm{row} }1:    &        y_1 = 2 ; &                 \\
\mbox{\rm{row} }i: & -1 = x_{i-1} y_{i-1}; & 3 = -x_{i-1} +y_{i}, \mbox{\, for\;} 2\leq i  \leq n-1;  \\
\mbox{\rm{row} }n:  & -1 = x_{n-1} y_{n-1}; & 2 = -x_{n-1} +y_{n}.
\end{array}
\]

\noindent Then, from $y_1=2$, we get $x_1= -1/2.$ Thus, 
$y_2 = 3+ x_1 = f_5/f_3$, and $x_2 =- 1/y_2= -f_3/f_5.$
Repeating recursively the process, we obtain:

For $2 \leq i \leq n-1$, 
\[
x_{i} =  - f_{2i-1}/f_{2i+1}, 
\]
and, using formula (\ref{Fibonacci})
\[
\begin{array}{ll} \vspace{0.2cm}
  y_i &= 3 + x_{i-1} = 3 - f_{2i-3}/f_{2i-1}= (3 f_{2i-1} -f_{2i-3})/f_{2i-1}\\ 
      &=(f_{2i-1}+f_{2i})/f_{2i-1}= f_{2i+1}/f_{2i-1},
\end{array}
\]
and 
\[
\begin{array}{ll}\vspace{0.2cm}
 y_n &= 2 + x_{n-1} = 2 - f_{2n-3}/f_{2n-1} 
 = (2 f_{2n-1} -f_{2n-3})/f_{2n-1}\\
     &= (f_{2n-1}+f_{2n-2})/f_{2n-1}= f_{2n}/f_{2n-1}.  
\end{array}
\]
\end{proof}

\begin{example}
{\rm 
For $n=4$, $\tilde{L}_{P_4}=L_{1}U$, with $L_1$ and $U$ as below. 
\[\left[\begin{array}{rrrr}
2&-1&0&0\\
-1&3&-1&0\\
0&-1&3&-1\\
0&0&-1&2
\end{array}
\right]=\left[\begin{array}{rrrr}
1&0&0&0\\
-\frac{f_1}{f_3}&1&0&0\\
0&-\frac{f_3}{f_5}&1&0\\
0&0&-\frac{f_5}{f_7}& 1
\end{array}
\right]\left[\begin{array}{rrrr}
\frac{f_3}{f_1}&-1&0&0\\
0&\frac{f_5}{f_3}&-1&0\\
0&0&\frac{f_7}{f_5}&-1\\
0&0&0&\frac{f_8}{f_7}
\end{array}
\right].
\]
}
\end{example} \endproof

Next we determine the inverses of $L$ and  $U$, showing nice expressions of these entries in terms of quotients of Fibonacci numbers.



\begin{proposition}
Let $U$ and $L_1$ be as in $(\ref{LandU})$. Then 
\begin{equation*}
U^{-1}= 
\left[ \begin{array}{ccccc}
\frac{f_1}{f_3} & \frac{f_1}{f_5}& \cdots & \frac{f_1}{f_{2n-1}} & \frac{f_1}{f_{2n}}\\
0 & \frac{f_3}{f_5}& \cdots&\frac{f_3}{f_{2n-1}} & \frac{f_3}{f_{2n}}\\
0 & 0 &  \ddots  & \vdots\\
0 & 0 & \ddots   & \frac{f_{2n-3}}{f_{2n-1}} & \frac{f_{2n-3}}{f_{2n}}\\
0 & 0 &  &  0 & \frac{f_{2n-1}}{f_{2n}}\\
\end{array}
\right] \mbox{\, and\, }
L_1^{-1}= 
\left[ \begin{array}{ccccc}
1 & 0&0&  \cdots   & 0\\
\frac{f_1}{f_3} & 1& 0& \cdots    & 0\\
\frac{f_1}{f_5} & \frac{f_3}{f_5}&1 & \cdots & 0\\
\vdots & \vdots & &\ddots   & \\
\frac{f_1}{f_{2n-1}} & \frac{f_{3}}{f_{2n-1}}  &\cdots &   \frac{f_{2n-3}}{f_{2n-1}}& 1\\
\end{array}
\right].
\end{equation*}

\end{proposition}

\begin{proof}
 The proof follows directly computing $UU^{-1}= I$ and $L_{1} L_{1}^{-1}= I$.
\end{proof}

We recall that we can also determine each column of the matrix $B_{P_n}$ using this decomposition solving the system $B_{P_n}x= e_i,$ which is equivalent to $(LU)x=e_i.$ We can find $x$ using the equations $Lz= e_i,$ and $Ux=z$, solving first $Lz=e_i$ to obtain $z$ and then $ Ux=z,$ to obtain $x$.

\section{Mononicity more generally}

We now consider more general graphs than paths, and we will see that several monotonicity  properties may be generalized to almost as strong inequalities. 
We also show  that some inequalities in (\ref{eq:d-mon-extra}) hold with equality. In fact, this is a consequence of the next result which is true for pendant vertices in general graphs.

\begin{lemma}
 \label{lem:pendant}
 Assume the graph $G$ has a pendant vertex $j$ which is incident to vertex $k$. Then the  inverse matrix $B_G=[b_{ij}]$ satisfies 
 \begin{equation}
 \label{eq:B_pendant}
 \begin{array}{cl} \vspace{0.1cm}
     b_{ik} = 2 \,b_{ij}   &(1 \le i \le n, \,i \not = j).
 \end{array}    
 \end{equation}
\end{lemma}

\begin{proof}
 Consider as above the column $i$ in both sides of the matrix equation $\tilde{L}_n B_G=I_n$,  so $\tilde{L}_n x=e_i$. Then row $i$ of $B_G$ equals $x^t$. The equation corresponding to vertex $v_j$ is 
 \[
    2 x_j-x_k=0,
 \]
so $x_k=2x_j$, and therefore $b_{ik} = 2 \,b_{ij}$.
\end{proof}

\begin{example} \label{inverseforpath}
{\rm We  compute 
\[
B_{P_4}=\tilde{L}_{P_4} ^{-1}=(1/21)\left[ \begin{array}{cccc}
13&5&2& 1\\
5&10&4&2\\
2&4&10&5\\
1&2&5&13
\end{array}
\right].
\]
Then, it is easy to see that $b_{13}= 2 b_{14},\; b_{23}= 2 b_{24}$ and $b_{32}= 2 b_{31},\; b_{42}= 2 b_{41}.$ \hfill{$\diamond$}
}
\end{example} 

\bigskip
We now generalize the monotonicity result (Theorem \ref{thm:B_monotone}) above to general trees.

\begin{theorem}
 \label{thm:tree_B_monotone}
  Let $T$ be a tree and let $B_T=[b_{ij}]=\tilde{L}_T^{-1}$. Let $v_i$ be a vertex and consider a path in $T$
  \[
    P=v_{i_0}, v_{i_1}, v_{i_2}, \ldots, v_{i_k}
  \]
  where $i_0=i$ and $v_{i_k}$ is a pendant vertex. Then row $i$ in $B_T$ 
   satisfies the following monotonicity property
  \begin{equation}
 \label{eq:tree-d-mon}
 \begin{array}{cl} \vspace{0.1cm}
     b_{i,i_t} \ge 2 \,b_{i,i_{t+1}}   &(0\le t\le k-1). 
 \end{array}    
 \end{equation}
\end{theorem}
\begin{proof}
 Let $x$ be column $i$ of $B_T$ and $x_j$ is the $j$'th component in $x$,  corresponding to vertex $v_j$. Then $x$ satisfies $\tilde{L}_Tx=e_i$. We  orient the edges of the tree such that $v_i$ is the root and every edge is directed away from the root. Such a directed edge is denoted by $(v_k,v_l)$, where $v_k$ is the tail and $v_l$ is the head. Then every non-root vertex has exactly one ingoing edge.
 
 We first establish a crucial property which is the basis for an induction proof, see Fig.~\ref{fig:subgr}. 
 
\smallskip
 {\em Observation:  Consider a vertex $v_k$, where $k\not=i$, with ingoing edge $(v_p,v_k)$ and outgoing edges 
 $(v_k,v_{j_r})$ for $r=1, 2, \ldots, s$ $($see Fig.~\ref{fig:subgr}$)$}.  Assume that $x_k \ge 2 x_{j_r}$ $(r \le s)$. Then $x_p \ge 2x_k$.
 
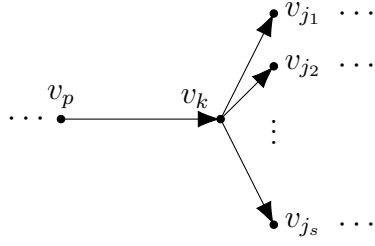
\begin{figure}

\centering
\begin{tikzpicture}[scale=0.7]
\coordinate (a) at (0,0);
\coordinate (b) at (3,0);
\coordinate (c) at (4,2);
\coordinate (d) at (4,1);
\coordinate (e) at (4,-2);
\coordinate (f) at (4,0.5);

\coordinate (g) at (5,2);
\coordinate (h) at (5,1);
\coordinate (i) at (5,-2);
\coordinate (j) at (0,0);

\draw (a) node[above]{$v_p$};
\draw (b) node[above left]{$v_k$};
\draw (c) node[right]{$v_{j_1}$};
\draw (d) node[right]{$v_{j_2}$};
\draw (f) node[below]{$\vdots$};
\draw (e)node[right]{$v_{j_s}$};

\draw (g)node[right]{$\cdots $};
\draw (h)node[right]{$\cdots $};
\draw (i)node[right]{$\cdots $};
\draw (j)node[left]{$\cdots $};

\draw [fill=black] (a) circle[radius=2pt];
\draw [fill=black] (b) circle[radius=2pt];
\draw [fill=black] (c) circle[radius=2pt];
\draw [fill=black] (d) circle[radius=2pt];
\draw [fill=black] (e) circle[radius=2pt];

\draw [-{Latex[length=3mm]}] (a)--(b) ;
\draw [-{Latex[length=3mm]}] (b)--(c) ;
\draw [-{Latex[length=3mm]}] (b)--(d) ;
\draw [-{Latex[length=3mm]}] (b)--(e) ;

\end{tikzpicture} 
\caption{Neighbors of vertex $v_k$.}
\label{fig:subgr}
\end{figure}

 \bigskip

 Proof of Observation: The equation in $\tilde{L}_Tx=e_k$ corresponding to vertex $v_k$  is 
 \[
     (s+2)x_k-x_p-\sum_{r=1}^s x_{j_r}=0.
 \]
  This gives 
  \[
  \begin{array}{rl} \vspace{0.2cm}
         x_p&=(s+2)x_k-\sum_{r=1}^s x_{j_r} \\ \vspace{0.2cm}
               &\ge (s+2)x_k-\sum_{r=1}^s (1/2)x_k \\ \vspace{0.2cm}
               &= (s/2+2)x_k  \\ \vspace{0.2cm}
               &\ge 2x_k
  \end{array}       
 \]
 which proves the observation. 
 
 Consider now a directed edge $(v_p,v_k)$ in $T$.  We say that this edge is {\em good} if $x_p \ge 2 x_k$.
 
\smallskip
 {\em Claim: Every directed edge $(v_p,v_k)$ in $T$ is good.}
 
 \smallskip
 Proof of Claim: If $(v_p,v_k)$ is such that $v_k$ is a pendant vertex, then, by Lemma
 \ref{lem:pendant}, $x_p = 2 x_k$, so $(v_p,v_k)$ is good. Next, consider an edge $(v_{\ell},v_p)$, such every outgoing edge of $v_p$ is good. Then, by the Observation, $x_\ell \ge  2 x_p$, so also the edge $(v_{\ell},v_p)$ is good. We continue like this, and can thereby continue until all edges are covered, and they are all good, as desired. 
 
 Finally, the Claim directly gives the monotonicity property (\ref{eq:tree-d-mon}), which completes the proof.
 \end{proof}

From this theorem we conclude that the entries in each row $i$ of $B_T$ drop off exponentially fast as a function of the (combinatorial) distance from the vertex $v_i$.  
 In the case of a path, from the proof of the Observation in the above proof, we have $x_p \geq \frac{5}{2} x_k.$ This lower bound is tighter to $x_p$ than the one found in Theorem \ref{thm:B_monotone}.

With these results in hand, we now turn our attention to the main diagonal of doubly stochastic matrices.

Consider a tree $T$. Define $d(v_i,v_j)$ as the (combinatorial) {\em distance} between vertices $v_i$ and $v_j$, i.e., the number of edges in the unique $v_iv_j$-path in $T$.

The next result provides a lower bound on the diagonal entries of doubly stochastic matrices constructed from trees, as below.

\begin{corollary}
 \label{cor:tree_B_monotone}
  Let $T$ be a tree with $n$ vertices and let $B_T=[b_{ij}]=\tilde{L}_T^{-1}$.  Then 
  \begin{equation}
   \label{eq:diag-L-ds}
     b_{ii} \ge \frac{1}{\sum_j 2^{-d(v_i,v_j)}}   \;\;\;\;(1\le i\le n). 
   \end{equation}
\end{corollary}
%
\begin{proof}
 We fix vertex $v_i$ as stated above. Consider a vertex $v_j$ and the (unique) $v_iv_j$-path
 \[
    v_{i_0}, v_{i_1}, v_{i_2}, \ldots, v_{i_k}
 \]
 where $k=d(v_i,v_j)$, and $i_0=i$, $i_k=j$. 
 Theorem \ref{thm:tree_B_monotone} with repeated application of the inequality (\ref{eq:tree-d-mon}) gives
 \[
      b_{ii} \ge 2 \,b_{i,i_1} \ge 2^2 \,b_{i,i_2} \ge \cdots \ge 2^k \,b_{i,i_k} .
 \]     
 Therefore 
 \[
    b_{ii} \ge 2^{d(v_i,v_j)} \,b_{ij} \;\;\;(1 \le j \le n).
 \]

 Since $B$ is doubly stochastic its row sums are 1, so
 \[
   1 = \sum_{j} b_{ij} \le \sum_j (2^{-d(v_i,v_j)}\, b_{ii})=b_{ii} \,\sum_j 2^{-d(v_i,v_j)}
 \]
 which gives the desired inequality.
\end{proof}
\medskip

\begin{example}
{\rm 
Let $G=Br(k, \ell)$ be the so-called broom tree which is obtained from a path $P_k$ with $k$ vertices by adding $\ell$ vertices and attaching each of these to the same end vertex of the path.
We let the vertices of $P_k$ be labeled $1, 2, \ldots, k$ and the remaining vertices are labeled  by $k+1, k+2, \ldots, k+\ell$. In Fig.~\ref{fig_broom_tree_revision} the graph $Br(6, 5)$ is shown.
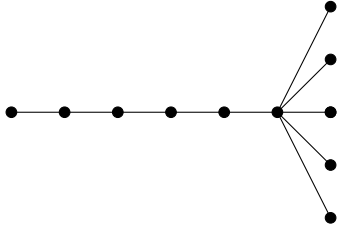
\begin{figure}[h]
\centering
\begin{minipage}{0.5\textwidth}
\centering

\begin{tikzpicture}[scale=0.7]
  \draw[fill] (0,0) circle(0.1cm);
  \draw[fill]  (1,0) circle (0.1cm);
  \draw[fill]  (2,0) circle (0.1cm);
  \draw[fill]  (3,0) circle (0.1cm);
  \draw[fill]  (4,0) circle (0.1cm);
   \draw[fill]  (5,0) circle (0.1cm);
  \draw[fill]  (6,0) circle (0.1cm);     
  \draw[fill]  (6,-1) circle (0.1cm);
  \draw[fill]  (6,0) circle (0.1cm);
 \draw[fill]  (6, 1) circle (0.1cm);
 \draw[fill]  (6, 2) circle (0.1cm);
 \draw[fill]  (6, -2) circle (0.1cm);

\draw(0,0) --(1,0);
\draw(1,0) --(2,0);
\draw(2,0) --(3,0);
\draw(3,0) --(4,0);
\draw(4,0) --(5,0);
\draw(5,0) --(6,0);
\draw (5,0)--(6,-1);
\draw (5,0)--(6,1);
\draw (5,0)--(6,2);
\draw (5,0)--(6,-2);

 \end{tikzpicture}
\caption{Broom tree $Br(6, 5)$.}
\label{fig_broom_tree_revision}
\end{minipage}
\end{figure}

The lower bound ($\ref{eq:diag-L-ds}$) in Corollary \ref{eq:diag-L-ds} for the diagonal entries of $B_G$ are 

\begin{eqnarray*}
b_{pp} & \geq & \frac{1}{3-(\frac{1}{2})^{p-1} +(\ell-2) (\frac{1}{2})^{k-p+1}
} \mbox{\, if \, \,} 1 \leq p\leq k,\\
b_{k+q, k+q}& \geq &\frac{1}{
\frac{7}{4} + \frac{\ell}{4} -(\frac{1}{2})^{k}}, \mbox{\, if \, \,} 1\leq q\leq \ell.
\end{eqnarray*}

Here the lower bound for $b_{pp}, 1\leq p\leq k$ comes from:
\begin{eqnarray*}    
\sum_j 2^{-d(v_p,v_j)}&= &
2^{-(p-1)}+ 2^{-(p-2)}+ \cdots+2^{-1}+2^0+ 2^{-1}+ 2^{-2}+ \cdots+ 2^{-(k-p)}+ \ell 2^{-(k-p+1)}\\
&=& 3-(\frac{1}{2})^{p-1}+ (\ell-2) \left(\frac{1}{2}\right)^{k-p+1}.\\
\end{eqnarray*}
Similarly, concerning the lower bound for $b_{k+q, k+q},$ with $1\leq q\leq \ell,$
we have 
\begin{eqnarray*}    
\sum_j 2^{-d(v_{p+q},v_j)}&= &
2^{-k}+ 2^{-(k-1)}+ \cdots+2^{-1}+2^0+ (\ell-1) 2^{-2}\\
&= & \frac{7}{4} -\left(\frac{1}{2}\right)^{k}+ \frac{\ell}{4}.
\end{eqnarray*}\hfill{$\diamond$}
}
    \end{example}

We now consider a general graph $G$. The goal is to establish a monotonicity property for the inverse matrix $B_G=\tilde{L}_G^{-1}$.  Let $d_j$ denote the degree of vertex $v_j$.

\begin{theorem}
 \label{thm:G_B_monotone}
  Let $G$ be a connected graph and let $B_G=[b_{ij}]=\tilde{L}_G^{-1}$. Consider a vertex $v_i$ in $G$. Let $v_j$ be a vertex different from $v_i$. 
  
  Then there exists a $v_jv_i$-path $P=v_j, v_{j_1}, v_{j_2}, \ldots, v_{j_k}, v_i$ such that the entries in row $i$ of $B_G$ are monotone in the following sense 
  \begin{equation}
    \label{eq:G-mon}
      b_{ij} < b_{i,j_1} < \cdots < b_{i,j_k} < b_{ii}. 
  \end{equation}

 In particular, $b_{ii}>b_{ij}$. So each diagonal element in $B_G$ is strictly largest in its row and column.
\end{theorem}
\begin{proof}
 Since $j \not = i$, the corresponding equation from $\tilde{L}_G x=e_i$ is 
 \[
     x_j=\frac{1}{d_j+1}\sum_{k:\, v_k \sim v_j} x_k. 
 \]
 Recall that  $v_k \sim v_j$ means that $v_k$ and $v_j$ are adjacent. Therefore  
 \[
     x_j=\frac{1}{d_j+1}\sum_{k:\, v_k \sim v_j} x_k <\frac{1}{d_j}\sum_{k:\, v_k \sim v_j} x_k,
 \]
  where the last expression is the average $x$-value among the neighbors of vertex $v_j$. (Note, concerning the inequality, we already know that every entry in the matrix $B$ is strictly positive.) Therefore there must be a  neighbor $v_{j_1}$ of $v_j$, i.e., $v_j \sim v_{j_1}$, with $x_j<x_{j_1}$.
 We repeat this calculation and argument with $j$ replaced by $j_1$. So there is a neighbor $v_{j_2}$ of $v_{j_1}$ such that $x_{j_1}<x_{j_2}$. Continuing like this we construct a path $P=v_j, v_{j_1}, v_{j_2}, \ldots, v_{j_k}, v_i$ with
  \[
   x_j<x_{j_1} < x_{j_2} < \cdots < x_{j_k}<x_i.
  \]
   All the vertices in $P$ must be distinct as the $x$-value is strictly increased in each step. The process can only terminate in vertex $v_i$ because in any other vertex we can find a neighbor with higher $x$-value. This proves the theorem. 
\end{proof}

Concerning Theorem \ref{thm:G_B_monotone} the inequality $b_{ii}>b_{ij}$ was also shown in \cite[Theorem 2]{Merris_II}  with a different argument (and the equality case was discussed). The  ``path inequalities''  (\ref{eq:G-mon}) give a stronger result of a more global character.

\section{Algorithm for trees and applications}
\label{sec:alg_tree}

The construction in the proofs of the previous subsection has a further potential which is to give a rather simple procedure for computing the inverse matrix $B_T=(\tilde{L}_T)^{-1}$ when $T$ is a tree. We discuss the details next. 

Let $T=(V,E)$ be a tree and label again its vertices by $v_1, v_2, \ldots, v_n$. Let $T^{(i)}$ be the directed tree obtained from $T$ by directing each edge away from the vertex $v_i$ ($i \le n$). We call $v_i$ the {\em root}  of $T^{(i)}$, and it is the unique vertex with no ingoing edge.
Let $\delta^+(v_k)$ denote the set of outgoing edges from vertex $v_k$, i.e., all tree edges of the form $(v_k,v_s)$ for some $s$. As usual $d_k$ is the degree of vertex $v_k$ in $T$.

Let $i \le n$ be fixed. The  following algorithm works on the tree $T^{(i)}$ with its root $v_i$. It computes edge labels, called {\em multipliers}, from the leaves towards the root $v_i$, and uses these multipliers to compute the $i$'th column $x$ in $B_T=(\tilde{L}_T)^{-1}$. 
For an edge $(v_p,v_k)$ the multiplier $m_{pk}$ will become 
 $m_{pk}=x_p/x_k$, so therefore  $x_p=m_{pk}x_k$.
The initial step is to consider the Laplacian equation for each pendent vertex, which relates the two values $x_p$ and $x_k$ where $v_k$ is the pendent vertex: $x_p=2x_k$. This gives the multiplier 2 for the edge $(v_p,v_k)$. 

Then, recursively,  we consider another edge $(v_p,v_k)$ where all multipliers have been determined for edges in $\delta^+(v_k)$, and consider the Laplacian equation for vertex $v_k$. This  gives a new multiplier for this edge $(v_p,v_k)$ which relates the values $x_p$ and $x_k$. Overall this gives a procedure where multipliers are computed from the leaves towards the root vertex $v_i$. Then the value $x_i$ may be determined (using that the row sum is 1), and, finally, all other values can be computed from the multipliers.  

\bigskip\bigskip\bigskip
{\small
\begin{tabbing}
{\bf Algor}\={\bf ithm $B_T$.} \\
Input: A tree $T$ and a vertex $v_i$ in $T$. 
\\
\\

 \>0. Construct the directed tree $T^{(i)}$.\\
 \>1. \=For each directed edge $(v_p,v_k)$ where $k\not =i$ and $v_k$ is a pendant vertex define the \\
   \>\>label $m_{pk}=2$. \\
 \>2. \>(a) Choose a directed edge $(v_p,v_k)$ where each edge in $\delta^+(v_k)$ has been \\
 \>\>labeled. Then give \=the edge $(v_p,v_k)$ the label \\
 \>\> \>         $m_{pk}=d_k+1-\sum_{j:(v_k,v_j)\in \delta^+(v_k)} 1/m_{kj}$. \\
 \>\>(b) Repeat step (a) until all edges have been labeled.  \\ 
\>3. (a) Let \\
\>\>\> $x_i=(d_i+1-\sum_{j:(v_i,v_j)\in \delta^+(v_i)} 1/m_{ij})^{-1}$. \\
\>\>(b) For  each directed edge $(v_p,v_k)$ where $x_p$ has been computed, let \\
\>\>\> $x_k=x_p/m_{pk}$.   \\

\\
Output: $x$, which is the $i$'th column  in $B_T=(\tilde{L}_T)^{-1}$. 
\end{tabbing}
}

\medskip
\begin{theorem}
 \label{thm:tree_alg_BT}
  Let $T$ be a tree with $n$ vertices. Algorithm $B_T$ correctly computes the $i$'th column $x$ in $B=(\tilde{L}_T)^{-1}$.  The number of steps is $O(n)$. This gives an $O(n^2)$ algorithm to compute $B_T$.
  
\end{theorem}
\begin{proof}
Let $x$ denote the $i$'th column of $B_T=(\tilde{L}_T)^{-1}$. Consider the directed tree $T^{(i)}$ and for each directed edge $(v_p,v_k)$ define 
 \[
    m^*_{pk}=x_p/x_k.
 \]
 
 We claim that $m_{pk}=m^*_{pk}$ for each $(v_p,v_k)$, where $m_{pk}$ is determined by Algorithm $B_T$. First, consider the case when $(v_p,v_k)$ is such that $v_k$ is a pendant vertex and $p\not =i$. Then, the Laplace equation associated with vertex $v_k$ is 
 $2x_k=x_p$, so 
 \[
    m^*_{pk}=x_p/x_k=2=m_{pk},
 \]
 as desired. Next, consider $(v_p,v_k)$ where $v_k$ is not a pendant vertex and $p \not =i$. Assume that $m^*_{kj}=m_{kj}$ for each $(v_k,v_j) \in \delta^+(v_k)$. Then the Laplace equation for vertex $v_k$ is 
 \[
 \begin{array}{ll}\vspace{0.2cm}
    (d_k+1)x_k&=x_p+\sum_{(v_k,v_j) \in \delta^+(v_k)} x_j\\ \vspace{0.2cm}
     &=x_p+\sum_{(v_k,v_j) \in \delta^+(v_k)} x_k/m^*_{kj} \\ \vspace{0.2cm}
     &=x_p+\sum_{(v_k,v_j) \in \delta^+(v_k)} x_k/m_{kj}.
 \end{array}
 \]
 Therefore
 \[
  \big(d_k+1-\sum_{(v_k,v_j) \in \delta^+(v_k)} 1/m_{kj}\big)x_k=x_p 
 \]
 which means, by Step 3(a) in Algorithm $B_T$ that
 $m_{pk}x_k=x_p$. So 
 \[
    m^*_{pk}=x_p/x_k=m_{pk}.
 \]
 Finally, consider vertex $v_i$. The corresponding Laplace equation gives the expression for $x_i$ in Step 3(a). Then, the remaining values $x_k$ ($k \not = i)$ are computed as in Step 3(b), as $m^*_{pk}=x_p/x_k=m_{pk}$.

\end{proof}
 
Thus Algorithm $B_T$ traverses the edge set of $T$ twice. The first time from leaves towards the root, in order to compute the multipliers $m_{kp}$. The second time one traverses from the root in, e.g.,  a breadth-first-search way and compute the desired vector $x$.

For small, or some simple classes of trees, Algorithm $B_T$ may be preformed analytically and provide an explicit expression for the inverse matrix $B_T$.
We give some such examples next.

\begin{example} 
{\rm 
We will apply Algorithm $B_T$ to the tree $T$ depicted in Figure \ref{tree} computing its $6$'th column. So, let $i=6$.
Step 2 gives $m_{52}=m_{53}=m_{41}=2$ and then $m_{65}=3$, $m_{64}=5/2$. In Step 3 we first compute $x_6=(3-(1/3+2/5))^{-1}=15/34$, and using the multipliers we find $x_{5}= 5/34$, $x_{4} =3/17$, $x_{2}= x_{3}=5/68$ and $x_{1} =3/34$. Thus the $6$'th column, and row,  of the matrix $B_{T}$ is $x=(3/34, 5/68, 5/68, 3/17, 5/34, 15/34)=\frac{1}{68}(6,5,5,12,10,30)$.
}\hfill{$\diamond$}
\end{example} 

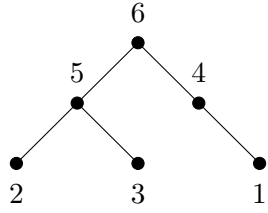
\begin{figure}
\centering
 \begin{tikzpicture}[scale=0.4]
  \draw[fill]  (4,0) circle (0.2cm);
  \draw[fill]  (2,-2) circle (0.2cm);
 \draw[fill]  (6,-2) circle (0.2cm);
 \draw[fill]  (0,-4) circle (0.2cm);
 \draw[fill]  (4,-4) circle (0.2cm);
 \draw[fill]  (8,-4) circle (0.2cm);
 \draw  (4,1) node{\small{$6$}};
  \draw  (2,-1) node{\small{$5$}};
  \draw (6,-1) node{\small{$4$}};
   \draw (0,-5) node{\small{$2$}};
   \draw (4,-5) node{\small{$3$}};
   \draw[fill] (8,-5) node{\small{$1$}};
 \draw (4,0) --(2,-2);
 \draw (4,0) --(6,-2);
 \draw (2,-2) --(0,-4);
 \draw (2,-2) --(4,-4);
 \draw (6,-2) --(8,-4);
 \end{tikzpicture}
\caption{A tree $T$ with $6$ vertices}
\label{tree}
\end{figure}

\begin{example}
{\rm    Let $S_n$ be the star with $n\geq 2$ vertices where the central vertex is labelled as $v_n$. Then  
    \begin{equation}\label{starinverse}
  B_{S_n}= \frac{1}{2n+2}\left[
  \begin{array}{ccccc}
    n+2& 1&\cdots& 1& 2 \\
   1& n+2&\cdots& 1& 2 \\
    && \ddots \\
    1& 1 &\cdots& n+2& 2 \\
   2& 2&\cdots& 2& 4 \\
  \end{array}
  \right].
  \end{equation}

To verify this we use Algorithm $B_T$ to compute the $i$'th column of $B_{S_n}$. 

First compute the last column of $B_{S_n}$. Step 1 gives $m_{nj}=2$ for $j\le n-1$, so Step 2 is not needed. Then Step 3 (a) gives  $x_n= (n-(n-1)(1/2))^{-1}= 2/ (n+1)$. Step 3 (b) gives $x_j=2/(2n+2)$ for $j<n$. 

Next,  compute the first column of $B_{S_n}$. Step 1 gives $m_{nj}=2$ for $2\le j \le n-1$, and Step 2 gives $m_{1n}=(n+2)/2$. Then Step 3 gives $x_1=(n+2)/(2n+2)$,  $x_n=2/(2n+2)$, and $x_j=1/(2n+2)$ for $2 \le j \le n-1$.
The other columns follow from column $1$, by symmetry.
}\hfill{$\diamond$}
\end{example}

Next, we consider paths. In the next example we will see that the entries of $B_{P_n}$ (in this case just for the last column) can be written as function of quotients of Fibonacci numbers.

\begin{example}
{\rm 
 Consider the path $P_n=v_1, v_2, \ldots, v_n$. We compute the $n$'th column of the matrix $B_{P_n}.$  
Step 1 gives $m_{21}= 2=\frac{f_3}{f_1}.$
From Step 2, and 
for each $2\leq j\leq n-1$, we have 
$$m_{j+1, j} = 3- \frac{f_{2j-3}}{f_{2j -1}}= \frac{3f_{2j-1}- f_{2j-3}}{f_{2j-1}} = \frac{f_{2j+1}}{f_{2j-1}}.$$

From Step 3 (a) we have 
$x_n = ( 2 - \frac{f_{2n-3}}{f_{2n-1}})^{-1}=( \frac{2 f_{2n-1} -f_{2n-3}}{f_{2n-1}})^{-1} =\frac{f_{2n-1}}{f_{2n}}.$
From Step 3 (b), using the multipliers, we obtain $$x_{n-1}=(m_{n,n-1})^{-1} x_n =\frac{f_{2n-3}}{f_{2n-1}}x_{n}=\frac{f_{2n-3}}{f_{2n}},$$ and so on. The procedure is repeated until $x_1$ is obtained. 
Then, the $n$'th column of $B_{P_n}$ is 
\[
x= (1/f_{2n}) \cdot   (f_1, f_3,\ldots , f_{2n-3}, f_{2n-1})    
\] 
as confirmed in Section 3.
}\hfill{$\diamond$}
\end{example}

\begin{example}
{\rm
Let $n_1, n_2,  \ldots, n_{k-1}$ be positive integers. We compute the first column of $B_{S},$ where $S=S(n_1,n_2, \ldots, n_{k-1})$ is a starlike tree.  The starlike tree $S=S(n_1,n_2, \ldots, n_{k-1})$ is a tree that results from the stars $S_{n_{1}+1}$, \ldots, $S_{n_{k-1}+1}$ by connecting their centers to an extra vertex with label $1$. The neighbours of $1$ are labelled as $2,3, \ldots , k$. 
For $j= 2, \ldots, k$ the neighbours of $j$ are the vertex $1$ and the pendant vertices labelled as $ k+n_{1}+\cdots+ n_{j-2} +1, \ldots,k+ n_{1}+n_{2} + \cdots +n_{j-1},$ where the sum $n_{1}+\cdots+ n_{j-2}$ is void for $j=2.$
See Fig.~\ref{fig:starlike} for $S(n_1, n_2, n_3)=S(3,3,4).$\\
Using Algorithm $B_T$ we will now compute the first column of $B_S$.
Thus for $j=2, \ldots, k$, Step 1 gives
$$m_{j, k+n_{1}+\cdots+ n_{j-2}+\ell_{j-1}}=2, \;\;(1\leq \ell_{j-1} \leq n_{j-1}), $$ where  $n_{1}+\cdots+ n_{j-2}$ is void for $j=2.$
\\

Step 2 gives: $ m_{1, \ell},$ with $2 \leq \ell \leq k$, 
$m_{1, \ell} = n_{\ell-1} +2 - (n_{\ell-1}) \frac{1}{2}= \frac{n_{\ell-1} +4}{2}.\bigskip$

From Step 3 (a) we have: 
$x_1 = (k- (\frac{2}{n_{1}+4} + \frac{2}{n_{2}+4}+\cdots+ \frac{2}{n_{k-1}+4}))^{-1}= \gamma.$
Then, from (b) we have $x_j =\frac{2}{n_{j-1}+4} x_{1}= \frac{2}{n_{j-1}+4} \gamma,$ with $2 \leq j\leq k,$
and
$$
\begin{array}{lllllllll}
x_{k+1}&= &x_{k+2}&= & \ldots&= &x_{k+n_1}&=& \frac{1}{n_1+4} \gamma;\\
x_{k+n_1 +1}& = &x_{k+n_1 +2}& =& \ldots &=& x_{k+n_1 +n_2}& = &\frac{1}{n_2+4} \gamma; \\
&\vdots& & & & & & &\\
x_{k+n_1 +\cdots+ n_{k-2}+1} &= &
x_{k+n_1 +\cdots+ n_{k-2}+2}& =& \ldots &=& x_{k+n_1 +\cdots+n_{k-2}+ n_{k-1}} &=& \frac{1}{n_{k-1}+4} \gamma.
\end{array}
$$
Thus, the 1st column of $B_{S}$ is:
$$
(\gamma, \frac{2}{n_1 +4}\gamma, \ldots, \frac{2}{n_{k-1} +4}\gamma,\underbrace{\frac{1}{n_1+4}\gamma, \ldots, \frac{1}{n_1+4} \gamma}_{n_1}, \ldots, \underbrace{\frac{1}{n_{k-1}+4} \gamma, \ldots, \frac{1}{n_{k-1}+4} \gamma}_{n_{k-1}} ).
$$
}
\end{example} \hfill{$\diamond$}

\begin{figure}
\centering
 \begin{tikzpicture}[scale=0.6]
  \draw[fill]  (0,0) circle (0.2cm);
  \draw[fill]  (-3,2) circle (0.2cm);
 \draw[fill]  (3,2) circle (0.2cm);
\draw[fill]  (0,-3) circle (0.2cm);
\draw[fill]  (0,-4) circle (0.2cm);
\draw[fill]  (1,-4) circle (0.2cm);
\draw[fill]  (-1,-4) circle (0.2cm);  
\draw[fill]  (4,3) circle (0.2cm);
\draw[fill]  (4,1) circle (0.2cm);
\draw[fill]  (4,2) circle (0.2cm);
\draw[fill] (3,3.5) circle (0.2cm);

\draw[fill]  (-4,3) circle (0.2cm);
\draw[fill]  (-3,3.5) circle (0.2cm);
\draw[fill]  (-4,1.5) circle (0.2cm);
 
 \draw  (0,0.6) node{$1$};
  \draw (0,0) --(-3,2);
  \draw (0,0) --(3,2);
  \draw (0,0) --(0,-3);
  
  \draw (0,-3) --(0,-4);
  \draw (0,-3) --(1,-4);
  \draw (0,-3) --(-1,-4);
    
  \draw (3,2) --(4,3);
  \draw(-3,2)--(-4,3);
  \draw(-3,2)--(-3,3.5);
  \draw(-3,2)--(-4,1.5);       
  \draw(3,2)--(4,1);
  \draw(3,2)--(3,3.5);  
    \draw(3,2)--(4,2); 
\end{tikzpicture} 
\caption{S(3,3,4)}
\label{fig:starlike}
\end{figure}

\medskip
We return to general trees and we will now give a rough estimate of the multipliers in Algorithm $B_T$. If $v_k$ is a pendent vertex, then its incident edge $(v_p,v_k)$ is called a {\em pendent edge}.  

\begin{theorem}
  \label{thm:mult-prop}
  Consider a tree $T$. Let $i \le n$ and consider the multipliers $m_{pk}$ in Algorithm $B_T$. For each edge $(v_p,v_k)$ the following inequalities hold 
  \begin{equation}
    \label{eq:B-ineq}
     2 \le (1/2)d_k+3/2 \le m_{pk} \le d_k+1.
  \end{equation}
  If $(v_p,v_k)$ is a pendent edge, equality holds throughout, so $m_{pk}=2$.
\end{theorem}
\begin{proof}
 If $(v_p,v_k)$ is a pendent edge, then $m_{pk}=2$ by Lemma \ref{lem:pendant}.  Next, consider an edge $(v_p,v_k)$ where $v_k$ is non-pendent, so $d_k\ge 2$.  
 Assume that 
 \[
     2 \le  m_{kj} \le d_j+1 \;\;\; \mbox{\rm for all $j$ with $(v_k,v_j) \in \delta^+(v_k)$.}
 \]
 Then, from Algorithm $B_T$, we have 
 \[
 \begin{array}{ll} \vspace{0.2cm}
    d_k+1\ge m_{pk}&=d_k+1-\sum_{j:(v_k,v_j)\in \delta^+(v_k)} 1/m_{kj} \\ \vspace{0.2cm}
       &\ge d_k+1-\sum_{j:(v_k,v_j)\in \delta^+(v_k)} 1/2 \\ \vspace{0.2cm}
       &= d_k+1-(1/2)(d_k-1) \\ \vspace{0.2cm}
       &=(1/2)d_k+3/2 \\ \vspace{0.2cm}
       &\ge 2.
 \end{array}   
 \]
 So, the result now follows by induction by considering edges in a sequence from pendent vertices up to the edge $(v_p,v_k)$.
\end{proof}

The previous theorem gives bounds on the change among the entries in $B=[b_{ij}]$ when one moves from a pendent vertex towards vertex $i$ (which gives number of row/column $i$ in $B$). We see from the proof that the lower bound on $m_{pk}$ is tight if and only if $v_j$ is a pendant vertex for each $(v_k,v_j) \in \delta^+(v_k)$.

\begin{corollary}
  \label{cor:mult-prop3}
  Let $T$ be a tree and let $i \le n$. Consider an edge $(v_p,v_k)$ in $T^{(i)}$ and define 
  \[
     m^*=\min\{m_{kj}: (v_k,v_j) \in \delta^+(v_k)\}.
  \]
  Then 
  \begin{equation}
  \label{eq:B-ineq2}
      m_{pk} \ge (1-1/m^*)d_k+1+1/m^*.
  \end{equation}

\end{corollary}

\begin{proof}
 As in the previous proof, from Algorithm $B_T$, we get 
 \[
 \begin{array}{ll} \vspace{0.2cm}
    m_{pk}&=d_k+1-\sum_{j:(v_k,v_j)\in \delta^+(v_k)} 1/m_{kj} \\ \vspace{0.2cm}
       &\ge d_k+1-\sum_{j:(v_k,v_j)\in \delta^+(v_k)} 1/m^* \\ \vspace{0.2cm}
       &= d_k+1-(1/m^*)(d_k-1) \\ \vspace{0.2cm}
       &=(1-1/m^*)d_k+1+1/m^*.
 \end{array}   
 \]
 \end{proof}

From Theorem \ref{thm:mult-prop} and Corollary \ref{cor:mult-prop3} we see that if the multipliers are ``large'' for all multipliers of outgoing edges of a vertex $v_k$, then the next multiplier $m_{pk}$ will be close to the degree $d_k$.

\medskip
As an application of the results above, we consider a special class of trees. Let $\mc{T}^{(3)}_n$ be the class of trees where each non-pendent vertex has degree $3$ (also called regular tree of degree $3$). These are a special case of {\em chemical trees}, i.e., trees that have no vertex with degree greater than $4$. Chemical trees are used to represent chemical structures and they offer a powerful way to understand and analyze chemical compounds.
Let $T \in\mc{T}^{(3)}_n$ and consider its vertices $v_1, v_2, \ldots, v_n$. Let $i,j \le n$ with $i \not = j$. We let $d(i,j)$ denote the distance between $v_i$ and $v_j$, defined as the number of edges in the (unique) $v_iv_j$-path. Also, let $V_p$ denote the set of pendent vertices in $T$. Define for $i \not = j$ 
\begin{equation}
 \label{eq:deg3-bd}  
\begin{array}{l}\vspace{0.4cm}
\tilde{\eta}(i,j)=
\left\{
\begin{array}{cr}
   (1/4)^{d(i,j)} &\mbox{if $v_j \not \in V_p$,} \\ 
   (1/2)\cdot (1/4)^{d(i,j)-1} &\mbox{if $v_j  \in V_p$;}
\end{array}
\right.
\\
\hat{\eta}(i,j)=
\left\{
\begin{array}{cr}
   (1/3)^{d(i,j)} &\mbox{if $v_j \not \in V_p$,} \\ 
   (1/2)\cdot (1/3)^{d(i,j)-1} &\mbox{if $v_j  \in V_p$.}
\end{array}
\right.
\end{array}
\end{equation}

\begin{corollary}
  \label{cor:3-trees}
  Let $T \in\mc{T}^{(3)}_n$ and consider $B_T=[b_{ij}]$. Let $i \le n$. Then 
  \begin{equation}
    \label{eq:T3-bounds} 
    \big(1+\sum_{j\not = i} \hat{\eta}(i,j) \big)^{-1} \le 
    b_{ii} \le \big(1+\sum_{j\not = i} \tilde{\eta}(i,j) \big)^{-1},
  \end{equation}
  and for $j \not = i$
    \begin{equation}
    \label{eq:T3-bounds2} 
    \tilde{\eta}(i,j)\big(1+\sum_{k\not = i} \bar{\eta}(i,k) \big)^{-1} 
    \le 
    b_{ij} \le \hat{\eta}(i,j)\big(1+\sum_{k\not = i} \tilde{\eta}(i,k) \big)^{-1}.
  \end{equation}
\end{corollary}

\begin{proof}
  Let $j \not = i$. Let the unique $v_iv_j$-path be 
  \[
    v_i, v_{i_1}, v_{i_2}, \ldots, v_{i_{t-1}}, v_j  
  \]
where $d(i,j)=t$. By the definition of multipliers $m_{pk}$ we obtain 
  \[
    b_{ij}=b_{ii} \cdot m_{ii_1}^{-1}m_{i_1i_2}^{-1} \cdots m_{i_{t-1}j}^{-1}  
  \]
 where we suppress the dependence of $j$ in the intermediate vertices.  Now, for an edge incident to a pendent edge, the multiplier is 2. Otherwise, as degrees for non-pendent vertices are $3$, it follows from Theorem \ref{thm:mult-prop} that
 \[
     3 \le m_{pk} \le 4.
 \]
 This gives
 \begin{equation}
  \label{eq:bii_rel}   
   1=\sum_j b_{ij}=b_{ii}+b_{ii}\sum_{j \not = i} m_{ii_1}^{-1}m_{i_1i_2}^{-1} \cdots m_{i_{t-1}j}^{-1}
 \end{equation}
 and by using the lower and upper bounds for the multipliers, one obtains (\ref{eq:T3-bounds}). Then  (\ref{eq:T3-bounds2}) also follows. 
\end{proof}

To close this section, we briefly discuss a centrality concept in graphs. 
As presented in \cite{Merris_II}, and proved in \cite{Chebotarev}, for a graph $G$ and the matrix $B_G=(\tilde{L}_G)^{-1}=[b_{ij}]$, the function 
\[
  \rho(v_i,v_j)=b_{ii}+b_{jj}-2b_{ij}
\]
gives a metric on $V$. Moreover (see \cite{Merris_II}),   
$r(i):=\sum_{j=1}^n \rho(v_i,v_j)$ ($i \le n$)
is minimized if and only if $b_{ii}$ is a minimum. Such a minimizing vertex  $v_i$  was called a {\em least remote vertex}. This suggests that diagonal entries that are small correspond to ``central'' vertices in the graph. 

For a tree $T$, this observation can also be seen heuristically from (\ref{eq:bii_rel}). In fact, small value of $b_{ii}$ corresponds to a large value of the last sum in (\ref{eq:bii_rel}) which (as $m_{pk}\ge 2$) typically happens for a vertex $v_i$ with ``many vertices close'' (close in terms of combinatorial distance in $T$). This is a property of central vertices in a tree. We leave it for further work to study these questions, and, hopefully, to prove interesting results relating $b_{ii}$ to combinatorial center concepts.

   \begin{figure}[ht]
    \begin{center}
    \vspace{-3.5cm}
	\includegraphics[width=10cm]{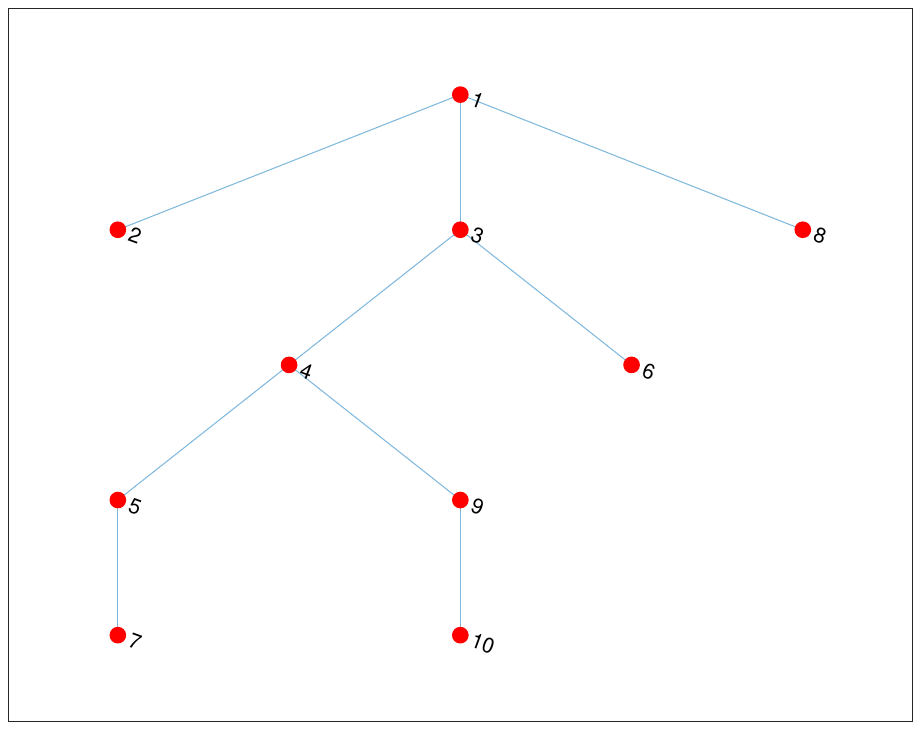}
    \vspace{-4.5cm}
    \end{center}
    \caption{A tree $T$.}
    \label{fig:tree1}
    \end{figure}
  
   \begin{example}
 \label{ex:tree_dist_bii}
    {\rm 
    Let $n=10$ and consider the tree $T$ in Fig.~\ref{fig:tree1}. Vertex $v_i$ is indicated by the label $i$. The matrix $B_T$ is shown below where the diagonal is indicated in boldface.
   
{\scriptsize 
\[
B_T =
\left[
\begin{array}{rrrrrrrrrr}
    {\bf 0.3723} & 0.1861 & 0.1168 & 0.0365 & 0.0146 &   0.0584 &   0.0073  &  0.1861&    0.0146&    0.0073\\
    0.1861 & {\bf 0.5931} & 0.0584 & 0.0182 & 0.0073 &   0.0292 &   0.0036  &  0.0931&    0.0073&    0.0036\\
    0.1168 & 0.0584 & {\bf 0.3504} & 0.1095 & 0.0438 &   0.1752 &   0.0219  &  0.0584&    0.0438&    0.0219\\
    0.0365 & 0.0182 & 0.1095 & {\bf 0.3467} & 0.1387 &   0.0547 &   0.0693  &  0.0182&    0.1387&    0.0693\\
    0.0146 & 0.0073 & 0.0438 & 0.1387 & {\bf 0.4555} &   0.0219 &   0.2277  &  0.0073&    0.0555&    0.0277\\
    0.0584 & 0.0292 & 0.1752 & 0.0547 & 0.0219 &  {\bf  0.5876} &   0.0109  &  0.0292&    0.0219&    0.0109\\
    0.0073 & 0.0036 & 0.0219 & 0.0693 & 0.2277 &   0.0109 &  {\bf  0.6139}  &  0.0036&    0.0277&    0.0139\\
    0.1861 & 0.0931 & 0.0584 & 0.0182 & 0.0073 &   0.0292 &   0.0036  & {\bf  0.5931}&    0.0073&    0.0036\\
    0.0146 & 0.0073 & 0.0438 & 0.1387 & 0.0555 &   0.0219 &   0.0277  &  0.0073&   {\bf  0.4555}&    0.2277\\
    0.0073 & 0.0036 & 0.0219 & 0.0693 & 0.0277 &   0.0109 &   0.0139  &  0.0036&    0.2277&  {\bf 0.6139}
\end{array}
\right].
\]
}
The smallest entries are $b_{44}$ followed by $b_{33}$. 
This suggests that these vertices play a role as ``central" vertices in the graph.

    }\hfill{$\diamond$}
\end{example}

\medskip
We remark that in \cite{GD07} distance vectors of  vertices in trees (vectors which contain the combinatorial distances of a vertex to all vertices in the tree) were used to define a partial order of vertices based on weak majorization of distance vectors. This was used to define a center concept, the majorization center, consisting of the ``smallest vectors'' in this majorization order. 
 In Example \ref{ex:tree_dist_bii} the majorization center consists of the vertices $v_3$ and $v_4$.

\section{Heat equation, extensions and concluding remarks}

We now briefly discuss a connection to partial differential equations (PDE) and the relevance of our results above for this. The {\em heat equation} \cite{Ragnar2005} is a PDE with applications in heat diffusion and particle diffusion, and it also plays a role in mathematical finance. The discrete heat equation has the form 
\[
\frac{du}{dt}=L_G u, \;\; u(0)=u_0,
\]
where $L_G$ is the Laplacian matrix of a graph $G$ obtained by discretizing the domain. Here  $u(x,t)$ indicates the temperature in position $x$ (in some space) at time $t$, and $u_0$ is the initial condition (a function of $x$ at time zero). For more on the heat equation and the computational approach to PDEs, see \cite{Ragnar2005}. 

The formal solution of the heat equation is 
$u(t)=P(t)u_0$ where $P(t)$ is the matrix exponential $P(t)=e^{tL_G}$. This  leads to different computational methods for solving the heat equation (e.g., Taylor expansion or eigenvector methods). Observe that 
\[
   P(t)=e^{tL_G}=e^{-tI}e^{t\tilde{L}_G}
\]
which means that the matrix exponential of the modified Laplacian matrix of $G$ may be used computationally as a heat solver, followed by a simple diagonal matrix scaling. Moreover, an extension of $\tilde{L}_G$  is of interest in connection with an implicit Euler method for solving the discrete heat equation, as we now briefly describe. Consider a step $h>0$ and consider the linear equation 
\begin{equation}
\label{eq:heat_eq_k}
    (I+hL_G)u^{k+1}=u^k   
\end{equation}
for finding a new iterate $u^{k+1}$ based on the previous iterate $u^k$. Here the coefficient matrix is 
\[
    \tilde{L}_G^{h}:=I+hL_G
\]
which is positive definite, and therefore invertible. So the solution of (\ref{eq:heat_eq_k}) is $u^{k+1}=(\tilde{L}_G^{h})^{-1}u^k$. Note that for $h=1$, we have $\tilde{L}_G^{1}=\tilde{L}_G$ and $(\tilde{L}_G^{1})^{-1}=B_G$. 

Thus, the new matrix $\tilde{L}_G^{h}$ generalizes the modified Laplacian matrix $\tilde{L}_G$, and it arises in a natural way from the heat equation.

 $\tilde{L}^{h}_G$ is an M-matrix, and has a nonnegative inverse. The inverse is positive if $G$ is connected.  Moreover, $\tilde{L}^{h}_G$ has constant row and column sums, equal to $1$, so (when $e$ is the all ones vector)  
\[
  \tilde{L}^{h}_G e=e, \;\;\mbox{\rm and}\;\; 
  e^t \tilde{L}^{h}_G=e^t.
\]
 If we multiply by $(\tilde{L}^{h}_G)^{-1}$ in the previous equations, we obtain 
\[
    (\tilde{L}^{h}_G)^{-1} e=e \;\;\mbox{\rm and}\;\; 
    e^t(\tilde{L}^{h}_G)^{-1}=e^t.
\]

This shows that $(\tilde{L}^{h}_G)^{-1}$ is a doubly stochastic matrix.

Next, we comment on how our previous results may be extended to $(\tilde{L}^{h}_G)^{-1}$ where the proofs are very similar:

\begin{itemize}
 \item Lemma \ref{lem:pendant}: replace the factor 2 in (\ref{eq:B_pendant}) by $\frac{1+h}{h}$.
 \item Theorem \ref{thm:tree_B_monotone}: replace the factor 2 in (\ref{eq:tree-d-mon}) by $\frac{1+h}{h}$.  
 \item Corollary \ref{cor:tree_B_monotone}: replace the factor 2  in (\ref{eq:diag-L-ds}) by $\frac{1+h}{h}$.
 \item Theorem \ref{thm:G_B_monotone}: the statement of the theorem is the same. For $h$ small this means that the values in the off-diagonal entries drop very fast as a function of the distance to $v_i$ (diagonal vertex).
 \item Theorem \ref{thm:tree_alg_BT}: as before. Algorithm $B_T$ is modified as follows: 
 in Step 1, $m_{pk} = (1+h)/h$; 
 in Step 2(a) the modified multiplier is 
 \[
    m_{pk}=1/h+d_k-\sum_{j:(v_k,v_j)\in \delta^+(v_k)} 1/m_{kj},
 \]
 and in 3(a) the modified expression for $x_i$ is 
 \[
    x_i=(1/h+d_i-\sum_{j:(v_i,v_j)\in \delta^+(v_i)} 1/m_{ij})^{-1}.
 \]
 \item Theorem \ref{thm:mult-prop}: replace (\ref{eq:B-ineq}) by 
 \[
      1+\frac{1}{h} \le (1-\frac{h}{h+1})d_k +\frac{h}{h+1} +\frac{1}{h}
      \le m_{pk} \le d_k+\frac{1}{h}. 
 \]
 \end{itemize}

\bigskip\bigskip
Finally, as possible future work, we believe that it is interesting to further explore the connection and usefulness of these modified Laplacians for PDEs. 
Another interesting topic is to study combinatorial and geometric properties of the set of doubly stochastic matrices obtained as inverses of modified Laplacian matrices, as a subset of $\Omega_n$.

\bigskip\bigskip
{\bf Acknowledgments.} We thank Kenneth H. Karlsen for pointing out the connections to the heat equation and, more generally, the role of the discrete Laplacian in PDEs.  
Enide Andrade is supported by CIDMA under the
Portuguese Foundation for Science and Technology 
(FCT, https://ror.org/00snfqn58)  
Multi-Annual Financing Program for R\&D Units with project UID/PRR/4106/2025.

\end{document}